\documentclass[9pt]{amsart}

\usepackage{amssymb,amsmath,amsthm,amsfonts}
\usepackage{setspace}
\usepackage[margin=1in]{geometry}
\usepackage{mathabx}
\usepackage[usenames,dvipsnames,svgnames,table]{xcolor}
\usepackage{enumerate}
\usepackage{comment}
\usepackage{latexsym}
\usepackage{hyperref}
\hypersetup{
  colorlinks   = true,    
  urlcolor     = blue,    
  linkcolor    = blue,    
  citecolor    = purple      
}

\newtheorem{thm}{Theorem}[section]
\newtheorem{prop}[thm]{Proposition}

\newtheorem{cor}[thm]{Corollary}

\newtheorem{remark}[thm]{Remark}

\newcommand{\R}{\mathbb{R}}

\newcommand{\Z}{\mathbb{Z}}
\newcommand{\T}{\mathbb{T}}

\newcommand{\lb}{\langle}
\newcommand{\rb}{\rangle}
\renewcommand{\t}{\widetilde}
\renewcommand{\d}{\,\operatorname{d}\!}

\newcommand{\W}{\operatorname{W}}
\renewcommand{\Re}{\operatorname{Re}}
\renewcommand{\hat}{\widehat}
\newcommand{\nhl}{N_{\text{HL}}}
\renewcommand{\ni}{N_{\text{I}}}
\renewcommand{\i}{\operatorname{I}\!}
\newcommand{\nm}{N}

\allowdisplaybreaks

\begin{document}

\title[Zakharov System on $\T$]{A Note on Global Existence for the Zakharov System on $\T$}

\author{E. Compaan}
\address{Department of Mathematics, Massachusetts Institute of Technology}
\email{compaan@mit.edu}

\begin{abstract}We show that the one-dimensional periodic Zakharov system is globally well-posed in a class of low-regularity Fourier-Lebesgue spaces. The result is obtained by combining the I-method with Bourgain's high-low decomposition method. As a corollary, we obtain probabilistic global existence results in $L^2$-based Sobolev spaces. We also obtain global well-posedness in $H^{\frac12+} \times L^2$, which is sharp (up to endpoints) in the class of $L^2$-based Sobolev spaces.
\end{abstract}

\maketitle


\section{Introduction \& Statement of Results}

In this note, we consider the periodic Zakharov system
\begin{equation*}
\begin{cases}
iu_t + \Delta u = nu, \qquad x \in\T, \; t \in \R, \\
n_{tt}-\Delta n = \Delta |u|^2,\\
u(x,0) = u_0(x),\\
n(x,0) = n_0(x), \quad n_t(x,0) = n_1(x).
\end{cases}
\end{equation*}
This system was introduced in the 1970s as a model of Langmuir turbulence in ionized plasma \cite{Zak}. The function $u$, the Schr\"odinger part, represents the envelope of a oscillating electric field, while the wave part $n$ represents the deviation from the mean of the ion density. The purpose of this note is to derive a global well-posedness result that holds for initial data at a low regularity level.

Recasting the Zakharov model as a first-order system by setting $n^\pm = n \pm i D^{-1} n_t$, where $D = (-\Delta)^\frac12$, we obtain
\begin{equation}\label{eq:ZS} 
\begin{cases}
iu_t + \Delta u = \frac12(n^+ + n^-)u, \qquad x \in\T, \; t \in \R, \\
in^\pm_t \mp Dn^\pm = \pm D|u|^2,\\
u(x,0) = u_0(x),\\
n^\pm(x,0) = n^\pm_0(x).
\end{cases}
\end{equation}
The Zakharov system \eqref{eq:ZS} conserves the Schr\"odinger part mass and a Hamiltonian energy:
\[ M(u) = \| u\|_{L^2}^2 \qquad \qquad E(u,n^\pm) = \| u_x \|_{L^2}^2 + \frac12 \Bigl( \| n^+\|_{L^2}^2 +  \| n^-\|_{L^2}^2 \Bigr)   + \frac12 \int (n^+ + n^-) |u|^2 \d x. \]
Note also that the mean values of $n$ and $n_t$ are preserved under the flow. Thus we may without loss of generality assume $n_0$ and $n_1$ are mean-zero. 
Sharp local well-posedness on $\T$ holds for initial data
\[ (u_0, n^\pm_0) \in H^\frac12(\T) \times L^2(\T). \]
This result is due to Takaoka \cite{Taka}, building on the result of Ginibre-Tsutsumi-Velo \cite{GTV} for the Euclidean space system. Well-posedness was proved used the Fourier restriction norm technique introduced in \cite{Bour2}.

The local result in sharp in the class of $L^2$-based Sobolev spaces. An earlier work of Bourgain \cite{Bour1} gives well-posedness in certain Fourier-Lebesgue type spaces. This result is neither stronger nor weaker than the Sobolev space result \cite{Taka}, in the sense than neither result implies the other. However, Bourgain's Fourier-Lebesgue space result also enables him to obtain a probabilistic global well-posedness result for \eqref{eq:ZS}. This will be discussed further below. 

The Hamiltonian conservation together with the Galgiardo-Nirenburg-Sobolev inequality yield global existence in the energy space $H^1(\T) \times L^2(\T)$. 

In the two-dimensional case, the sharp Sobolev space local theory remains at the same level -- it holds for initial data $(u_0, n^\pm_0) \in H^\frac12(\T^2) \times L^2(\T^2)$. However, to prove this, Kishimoto employed modified Besov-type Fourier restriction norms \cite{Kish2}. Similar arguments had previously been used by Bejenaru-Herr-Holmer-Tataru \cite{BHHT} to prove the parallel result for the Zakharov system on $\R^2$. We will use these norms as well in our work on $\T$. 

In the $\T^2$ setting, Kishimoto has also obtained global existence below the energy space, in $H^{\frac9{14}+}(\T^2) \times L^2(\T^2)$ for Schr\"odinger data with sufficiently small $L^2(\T^2)$ norm (in relation to that of the ground-state solution) \cite{Kish2}. The result is established using the I-method of Colliander-Keel-Staffilani-Takaoka-Tao \cite{CKSTT}. This result also applies to $\T$, since a solution to \eqref{eq:ZS} on $\T$ corrresponds to a solution on $\T^2$ which is constant along one spatial dimension. The constraint on norm of the Schro\"odinger part in $L^2$ is necessary on $\T^2$ to ensure that the Hamiltonian constrols the relevant Sobolev norm. 

In the one-dimensional case, Bourgain gave the following probabilistic global well-posedness result for \eqref{eq:ZS}:
\begin{thm}[\cite{Bour1}] \label{BourThm}
There are Sobolev exponents $0 < \sigma < s < \frac12 < \mu <1$ such that the Zakharov system \eqref{eq:ZS} is well-posed for data $(u_0, n_0, n_1)$ satisfying
\[ u_0 \in H^s(\T), \qquad \sup_k |k|^{\mu} | \hat{u_0}(k)| < \infty,\]
\[\sup_k |k|^{-\sigma} | \hat{n_0}(k)| < \infty, \qquad \sup_k | k |^{-\sigma - 1} | \hat{n_1}(k)| < \infty.\] Moreover, the system is almost surely globally well-posed  with respect to the normalized Gibbs measure associated with the Zakharov equation, which is supported on  $ \bigcap_{s < \frac12}\Bigl(H^s \times H^{s-1} \times H^{s-2}\Bigr)$. 
\end{thm}
An examination of Bourgain's proof reveals that one should take $\sigma$ very close to zero, and $s$ and $\mu$ very close to one half. The almost-sure global existence was established by recasting the equation as an infinite dimensional Hamiltonian system for the spatial Fourier coefficients. The global well-posedness holds almost surely with respect to the Gibbs measure, a Sobolev space probability measure based on the Hamiltonian of the equation. This approach was inspired by Lebowitz-Rose-Speer \cite{LRS}. See also \cite{Bour3} for a similar argument for the nonlinear Schr\"odinger equation. 

Thus, global existence is known deterministically in $H^{\frac9{13}+}(\T) \times L^2(\T) $ and probabilistically in $H^{\frac12-}(\T)  \times H^{-\frac12-}(\T) $. The object of this note is to prove a global existence result which, at least partially, fills this gap. 
Note that one cannot simply use a preservation of regularity argument to conclude some sort of global existence between these spaces. This is because the support of the Gibbs measure is in $H^{\frac12-}\backslash H^{\frac12} \times H^{-\frac12-} \backslash H^{-\frac12}$; that is, the probabilistic result does not give any clear insight into the dynamics in smoother spaces.

The idea of our proof is to combine the high-low decomposition method of Bourgain \cite{Bour4} with the I-method of Colliander-Keel-Staffilani-Takaoka-Tao \cite{CKSTT}. This combination has appeared in the work of Bourgain for the quintic Schr\"odinger equation on $\T$ \cite{Bour5}. We note that we were motivated in part by the probabilistic global existence result of Colliander-Oh \cite{CoOh}, which employed the high-low method. We now give an outline of the method.

The high-low method works by breaking the initial data into high-frequency and low-frequency parts, at some cut-off value $N$. One then solves the system with the low-frequency (and hence smooth) data in the energy space. To obtain a solution to the system with the full initial data, one solves a difference equation for the remainder, with the high-frequency initial data. To iterate, it must be proved that the nonlinear part of the solution to the difference equation is in the energy space (i.e. smoother than the initial data) with small norm. One can then add this smoother part to the low-frequency solution and repeat the process. If one can obtain strong enough bounds, any time interval $[0,T]$ can be covered via iteration, by choosing $N$ large enough. 

The I-method works by applying a smoothing operator, again with some cut-off parameter $N$, to the system. A local theory is then obtained for this smoothed version of the system in the energy space. The next step is to show that a modified energy corresponding to the smoothed system is ``almost conserved", i.e. it doesn't not grow too rapidly. If the growth is slow enough, one can again iterate to cover any interval $[0,T]$ by choosing a sufficiently large cut-off value $N$. 

In general, the I-method will yield well-posedness at lower Sobolev regularities than the high-low method.  However, in our case the I-method cannot be directly applied. It requires a local theory in $L^2$-based Sobolev spaces whose norms can be controlled by the modified energy. However, there is no such local theory for wave data below $L^2$ for the Zakharov system. We therefore must subtract out the linear flow in order to solve in Sobolev spaces. This leaves a difference equation, which is no longer  expected to satisfy an almost-conservation law. 

On the other hand, if we attempt to apply the high-low method alone, we are required  to obtain estimates of the form
\[ \| u W(t)(n_0^\pm) \|_{X^{1, -\frac12, 1}} \lesssim \|n_0^\pm\|_{Y^{r,\frac12,1}} \| u \|_{X^{1, \frac12, 1}}.\] 
Here $W(t)(n_0^\pm)$ is the linear wave flow. The norms which appear here are defined below. The relevant point for the moment is that the left-hand side is a norm with Sobolev regularity one, while the wave flow is rougher than $L^2$. For such rough wave data, it seems impossible to obtain the necessary estimate. Randomization of the wave data does not appear to improve matters. In the $H^s \times H^r$ local theory, this case corresponds to the (unattainable) $s > r +1$ regime.  

In our work, we begin by splitting data into high- and low-frequency parts. For the low-frequency part, we apply an I-method local theory, and control growth via an almost-conservation law. The remainder solves a difference equation. We show that the nonlinear part of this difference equation lives in the space where the I-method applies, with small norm. We use this smallness with the almost-conservation law to iterate.

We take initial data of the following form. Fix $u_0 \in H^s(\T)$, for $s \in (\frac12, 1)$, and wave initial data such that:
\begin{align}\label{eq:wave_data}
n_0^\pm(x) = \sum_{k \in \Z/\{0\}} \frac{h_k^\pm}{\lb k \rb^\beta} e^{ikx} \qquad \text{with} \qquad \sup_k |h^\pm_k| < \infty.
\end{align}
That is, we take $n_0^\pm$ in the Fourier-Lebesgue space
\[ \mathcal{FL}^{\beta, \infty} := \Bigl\{ f \in \mathcal{D}'(\T) \; | \; \sup_k\, \lb k \rb^\beta |\hat{f}(k)| < \infty \Bigr\}.\] 
We fix $\beta \in (0, \frac12]$, so that $n_0^\pm  \in  H^r(\T)$
for any  $ r < \beta -  \frac12$. We require $\beta > s - \frac12$. This corresponds to requiring that $n_0^\pm$ is in a Sobolev space $H^r$ with $s-1 < r$. Note that the case $s-1 = r$ is critical in some sense. When $s-1 > r$, it seems impossible to obtain any well-posedness. Our main result is as follows.

\begin{thm} \label{mainThm}
Suppose $u_0 \in H^s(\T)$ and $n_0^\pm \in \mathcal{FL}^{\beta, \infty}$, with 
\[ s> \frac12 \quad \text{ and } \quad \beta >\frac1{2(2-s)}  . \]
Then for any $T >0$, the Zakharov system \eqref{eq:ZS} with initial data $(u_0, n_0^\pm)$ has a solution on the time interval $[0,T]$. Furthermore, the norm of the nonlinear part of the solution grows at most polynomially in time. In particular,
\begin{align*}
    max _{0 \leq t \leq T} \Bigl( \| u - e^{it\delta}u_0 \|_{H^s} + \| n^\pm - e^{\pm i t \partial_x} n_0^\pm \|_{L^2} \Bigr) \lesssim C(\|u_0\|_{H^s}, \|n_0^\pm\|_{\mathcal{FL}^{\beta, \infty}}, s, \beta) \lb T\rb ^{\max\{ \alpha(1-s)\gamma, (\frac12 - \beta) \gamma\} }, 
\end{align*}
where 
\[  \alpha \in \Biggl( 1 - 2\beta   , \min\Bigl\{\frac{\beta + \frac12 - s }{2(1-s)}, \; \frac{2\beta - s   }{1-s}  \Bigr\} \Biggr)\] 
and 
\[ \gamma > \max\Bigr\{ \frac{1}{(2s-1)\alpha}, \frac{1}{\beta + \frac12 - s - 2\alpha(1-s)}, \frac1{\alpha + 2\beta -1   },  \frac{1}{2\beta - s    - \alpha(1-s)} \Bigl\}.  \]
\end{thm}

\begin{remark}
	It is probable that our methods would also apply to rougher Schr\"odinger data (that is, to $ u_0 \in \mathcal{FL}^{\alpha, \infty}$, with $\alpha \leq 1$). However, we do not pursue this here. 
\end{remark}

\begin{remark}We note that Bourgain gives a local theory result for wave data in a subset of $H^{-\sigma}$, with $0 < \sigma \ll 1$. We provide an alternate proof of this, which allows $0 < \sigma < \frac16$ and is couched in terms suited to our proof of deterministic global existence. 
\end{remark}

The theorem above implies a probabilistic well-posedness result as well. To state it, we first recall the definition of the Gaussian measure associated to $H^s(\T)$. The measure is given by 
\[ \d \mu_s = Z_s^{-1} e^{-\frac12 \| n\|_{H^s}^2} \prod_{x \in \T} \d n(x). \]
A typical element in the support of the measure is of the form
\[ n = n^\omega = \sum_{k \in \Z} \frac{g_k(\omega)}{\lb k \rb^s} e^{ikx}, \]
where $\{g_k(\omega)\}_k$ are independent standard Gaussian random variables. We see that the $n$ is $\omega$-almost-surely in $H^{s-\frac12-}\backslash H^{s-\frac12}$, so the measure is supported on $\bigcap_{r< s - \frac12} H^{r}(\T)$. This leads to the following corollary.

\begin{cor}
The Zakharov system \eqref{eq:ZS} is almost-surely globally well-posed for initial data in the space $H^s \times H^r$ endowed with the Gaussian probability measure if
\[ s> \frac{1}{2} \quad \text{and} \quad r >  -\frac{1-s}{2(2-s)}. \] 
In particular, we have global existence almost-surely in $H^{\frac12+} \times H^{-\frac1{6}+}$.
\end{cor}

The corollary follows by noting that typical data  in $H^r$ is of the form
\[ n_0 = n^\omega_0 = \sum_{k \in \Z} \frac{g_k(\omega)}{\lb k \rb^{r + \frac12 + \epsilon}} e^{ikx}, \]
and thus its Fourier coefficients satisfy
\[ \sup_k \lb k \rb^{ r + \frac12}|\hat{n_0}(k)| < \infty \]
almost surely. Hence the conditions of Theorem \ref{mainThm} are met for $r = \beta - \frac12$. 

As a corollary of the proof of the main theorem, we also have the following Sobolev space global well-posedness result, which matches the sharp local theory up to the endpoint. 

\begin{cor}
The one-dimensional Zakharov system \eqref{eq:ZS} is globally well-posed for initial $(u_0, n_0^\pm) \in H^{s}(\T) \times L^2(\T)$ for any $s > \frac12$.
\end{cor}

We will work with the Besov-type Fourier restriction norms defined by 
\begin{align*}
\| u \|_{X^{s,b, 1}} &= \left\| N^s L^b \| P_{N,L}u\|_{L^2_xL^2_t} \right\|_{\ell^2_N \ell^1_L}, \\
\| n \|_{Y^{s,b, 1}_\pm } &= \left\| N^s L^b \| Q^\pm _{N,L}n\|_{L^2_xL^2_t} \right\|_{\ell^2_N \ell^1_L}, 
\end{align*}
where $N\gtrsim 1$ and $L \gtrsim 1$ are dyadic and the frequency restriction operators are 
\begin{align*}
    \hat{P_{N,L}u}(k, \tau) &:= \chi_{|k| \approx N} \chi_{|\tau - k^2| \approx L } \hat{u}(k,\tau) \\
    \hat{Q^\pm_{N,L}n}(k, \tau) &:= \chi_{|k| \approx N} \chi_{|\tau \pm |k|| \approx L } \hat{n}(k,\tau). 
\end{align*} 
We also have time-localized versions of these spaces, denoted $X^{s,b, 1}_\delta$ and $Y^{s,b,1}_{\pm \delta}$, which are defined in the usual fashion. These norms were used for the Zakharov system in \cite{BHHT}; their properties may be found there. In the following, we often drop the $\pm$ from the wave part notation for simplicity.

\section*{Acknowledgments}
The author is grateful to Professor G. Staffilani for many helpful discussions. This work was supported by  NSF MSPRF \#1704865.

\section{Proof of Theorem \ref{mainThm}}

\begin{proof}
The idea of the proof is to combine the high-low method of Bourgain with an I-method argument. In the following, $\nhl \gg 1$ is the frequency cut-off associated with the high-low method part of the argument, and $\ni \gg 1$ is that associated with the I-method portion of the argument. These parameters will be fixed later. 

Begin by splitting the initial data $n_0^\pm$ into high and low parts:
\begin{align*}
n^\pm_0(x) =  P_{\leq \nhl} n_0^\pm + P_{n > \nhl}n_0^\pm =: n_0^{\pm \text{L}}(x) + n_0^{\pm \text{H}}(x).
\end{align*}
Here $P_{\leq \nhl}$ denotes the spatial frequency projection onto frequencies of magnitude at most $\nhl$, i.e.
\[ \hat{P_{\leq \nhl} f} := \chi_{ \{|k| \leq \nhl\}} \hat{f} \quad \text{and} \quad  P_{> \nhl} := \operatorname{Id} - P_{\leq \nhl}. \]

Let $(u^1,n^{\pm1})$ be the solution to \eqref{eq:ZS} with initial data $(u_0, n_0^{\pm \text{L}})$. To obtain this solution, we use the local theory adapted to the I-method from \cite[Prop. 4.5]{Kish2}. This says that we can obtain a solution $(u^1,n^{\pm1}) \in X^{s,\frac12,1}_\delta \times Y^{0,\frac12,1}_{\pm,\delta}$ on $[0,\delta]$, for 
\begin{equation*}
\label{eq:deltaDef} 
\delta \approx \left( \| \i u_0 \|_{H^1} + \| n_0^{\pm \text{L}} \|_{L^2} \right) ^{-2-}. \end{equation*}

Here the smoothing operator $\i: H^s \to H^1$ is defined by
\begin{equation*}
    \hat{\i f}(\xi) = m(\xi) \hat{f}(\xi), \quad \text{where} \quad m(\xi) = \begin{cases}
     1 &\text{for  } |\xi| <  \ni, \\
     (\ni/|\xi|)^{1-s} &\text{for  } |\xi| > 2\ni.
    \end{cases}
\end{equation*}
The multiplier $m$ is taken to be smooth and non-increasing in $|\xi|$. The multiplier, and hence $\i\;$ itself, is dependant on the choice of $\ni$, but we elect to suppress this in the notation for simplicity. 

For initial data satisfying $\| u_0\|_{H^s} = K$ and $\sup_k | \hat{n_0^{\pm \text{L}}}(k)| \leq K \lb k \rb^{-\beta}$, we may choose \footnote{If $\beta = \frac12$, we must include an additional factor of $\log \nhl$. This is harmless, so the case is ommitted for simplicity.} 
\begin{equation*} 
\delta \approx  \Bigl[  K \ni^{1-s} +  K \nhl^{\frac12 -\beta} \Bigr]^{-2-} =: N^{-2-}. 
\end{equation*}
and obtain the bound
\begin{equation}\label{eq:normBound} \delta^{\frac12-} \left( \| \i u^1 \|_{X^{1,\frac12,1}_\delta} + \| n^{\pm1} \|_{Y^{0,\frac12,1}_{\pm,\delta}} \right) \lesssim 1. 
\end{equation}
To continue the argument, we need a local theory for the difference equation which results when we subtract this solution $(u^1,n^{\pm1})$ from the solution for the Zakharov system \eqref{eq:ZS} with the full initial data $(u_0,n_0)$. This local theory is given by Proposition \ref{localtheory} below. Essentially, it says that as long as $|\hat{n_0^{\pm \text{H}}}(k)| \leq C_1 \lb k \rb^{-\beta}$, the difference equation can be solved on $[0,\delta]$, and the $H^s \times L^2$ norm of its nonlinear part is at most order $\nhl^{s-\beta - \frac12  }$. That is, the nonlinear part is small and smooth. \footnote{This holds as long as $\delta$ is sufficiently small. Specifically, it is necessary that $\delta^{\epsilon_0-} C_1 \lesssim 1$, where $\epsilon_0 \ll 1$ is the implicit constant in the exponent of \eqref{eq:normBound}. That is, $\delta^{\frac12-\epsilon_0} \Bigl( \| \i u^1 \|_{X^{1,\frac12,1}_\delta} + \| n^{\pm1} \|_{Y^{0,\frac12,1}_{\pm,\delta}} \Bigr) \lesssim 1 $. This is possible to achieve for arbitrarily large $K$ and $C_1$ and arbitrarily small $\epsilon_0$ by taking $\ni$ and $\nhl$ sufficiently large.} 

We add this nonlinear part $(v^1, \t{m}^{\pm1})$ to $(u^1(\delta), n^{\pm1}(\delta))$ and evolve according to \eqref{eq:ZS} again to obtain functions $(u^2,n^{\pm2})$ which solve \eqref{eq:ZS} on the interval $[\delta, 2 \delta]$. At time $2\delta$, we again add in the smooth part $(v^2, \t{m}^{\pm2})$ of the solution to the difference equation, and again evolve according \eqref{eq:ZS} to obtain $(u^3, n^{\pm3})$. This can continue as long as the norm of the solution does not grow too much. To control the growth, we carefully choose the frequency thresholds $\nhl$ and $\ni$. 

To understand the growth of the norm, we study the Hamiltonian energy. The growth of the Hamiltonian is bounded as follows. Recall that $(u^j,n^{\pm j})$ is defined on the interval $[(j-1)\delta, j \delta]$ with \[(u^{j}, n^{\pm j })(\delta(j-1)) = \Bigl(u^{j-1} + v^{j-1}, n^{\pm j-1 } + \t{m}^{\pm j-1 }\Bigr)(\delta(j-1)) .\]
Then the growth of the Hamiltonian is bounded by
\begin{align*}
     \Bigl|H(\i u^J, n^{\pm J})(\delta J) - H(\i u^1, n^{\pm 1} )(0)\Bigr| 
    &\leq 
    \Bigl| H(\i u^J, n^{\pm J})(\delta J) - \t{H}( u^J, n^{\pm J})(\delta J) \Bigr| \\
    & + \sum_{j=1}^J \Bigl| \t{H}( u^j, n^{\pm j})(\delta j) - \t{H}( u^j, n^{\pm j})(\delta (j-1)) \Bigr| \\
    & + \sum_{j=1}^{J-1}\Bigl|  \t{H}(u^{j} + v^{j}, n^{\pm j } + \t{m}^{\pm j }) (\delta j) -\t{H}(u^{j}, n^{\pm j } ) (\delta j) \Bigr|\\
    & + \Bigl| \t{H}( u^1, n^{\pm 1})(0) - {H}(\i u^1, n^{\pm 1})(0) \Bigr|, 
\end{align*}
where $\t{H}$ is the modified Hamiltonian given in \cite[Section 3]{Kish2}. It is defined by
\[ \t{H}(u,n^\pm) = \| \i u \|_{\dot{H}^1}^2 + \frac12 \| n^\pm \|_{L^2}^2 + \frac12 \sum_{\sum k_j =0} \hat{u}(k_1) \hat{\overline{u}}(k_2) \Bigl[ \sigma^+(k_1,k_2) \hat{n}^+(k_3) +\sigma^-(k_1,k_2) \hat{n}^-(k_3) \Bigr] .\] 
Here $\sigma^\pm$ is the bounded nonsingular Fourier multiplier operator defined in \cite[(3.2)-(3.4)]{Kish2}. The effect of the multiplier $\sigma^\pm$ is to eliminate certain low-order terms from $\frac{\d}{\d t} \hat{H}(u,n^\pm)$, resulting in a more favorable growth bound for the modified Hamiltonian. The precise definition of $\sigma^\pm$ is rather lengthy (and not used here) so we omit it. 

We proceed by bounding the growth of each term in the sum above. By \cite[Prop. 3.2]{Kish2}, for $s > \frac12$, we can bound the difference between the Hamiltonian energy and the modified Hamiltonian by 
\begin{multline*}
    \Bigl| H(\i u^J, n^{\pm J})(\delta J) - \t{H}( u^J, n^{\pm J})(\delta J) \Bigr| 
     + \Bigl| \t{H}( u^1, n^{\pm 1})(0) - {H}(\i u^1, n^{\pm 1})(0) \Bigr| \\
     \lesssim \ni^{-1+}  \Bigl( \| \i u^J(\delta J) \|_{H^1}^2 \| n^{\pm J}(\delta J) \|_{L^2} + \| \i u_0 \|_{H^1}^2 \| n^{\pm }_0 \|_{L^2} \Bigr) \lesssim \ni^{-1+} \nm^3.
\end{multline*}
 
Next we bound the growth of the modified energy under the Zakharov flow. We use the following estimate. It is almost identical to the $\T^2$ bound \cite[Prop. 4.1]{Kish2}. However, it is slightly stronger because of more favorable estimates available in the one-dimensional case. The proof is in Section \ref{prf_of_growth_bound}. 

\begin{prop}\label{growth_bound}
Fix $\frac12 < s < 1$ and $\delta \in (0,1)$. Suppose $(u, n^\pm)$ is a smooth solution to \eqref{eq:ZS} on the time interval $[0,\delta]$. Then
\begin{multline*}
    | \t{H}(u, n^\pm)(\delta) - \t{H}(u, n^\pm)(0)| \lesssim 
    \ni^{-1+} \delta^{\frac12 -} 
    \| \i u \|_{X^{1,\frac12,1}_\delta}^2 
    \| n^\pm \|_{Y^{0,\frac12,1}_{\pm\delta}} \\
    + \Bigl[ \ni^{-2+} 
    + \ni^{-\frac32+} \delta^{\frac12 -} 
    + \ni^{-1} \delta^{1-}\Bigr] 
    \left( \| \i u \|_{X^{1,\frac12,1}_\delta}^4  
    + \| \i u \|_{X^{1,\frac12,1}_\delta}^2 
    \| n^\pm \|_{Y^{0,\frac12,1}_{\pm\delta}}^2 \right).
    \end{multline*} 
\end{prop}

This gives 
\begin{align*}
    \Bigl| \t{H}( u^j, n^{\pm j})(\delta j) - \t{H}( u^j, n^{\pm j})(\delta (j-1)) \Bigr| 
    \lesssim& 
    \ni^{-1+} \delta^{\frac12-} \nm^3  + 
    \Bigl(\ni^{-2+} + \ni^{-\frac32+} \delta^{\frac12-} + \ni^{-1+} \delta^{1-}\Bigr)N^4 \\
    \approx &
    \ni^{-1+} \nm^2  + 
    \ni^{-2+}\nm^4 + \ni^{-\frac32+}\nm^{3}.
\end{align*}

Finally, we use the definition of the modified energy to note that 
\begin{align*}
    \Bigl|  \t{H}(u^{j} + &v^{j}, n^{\pm j } + \t{m}^{\pm j }) (\delta j) -\t{H}(u^{j}, n^{\pm j } ) (\delta j) \Bigr| \\
    \begin{split}
    &\lesssim 
    \| \i v^j \|_{\dot{H}^1} \| \i ( u^j + v^j) \|_{\dot{H}^1} + \| \t{m}^{\pm j}\|_{L^2} \|  n^{\pm j} + \t{m}^{\pm j}\|_{L^2} \\
    &\quad + \| u^{j} + v^{j} \|_{L^2} \| u^{j} + v^{j}\|_{H^{\frac12+}} \| \t{m}^{\pm j} \|_{L^2} + \Bigl( \| u^{j} + v^{j}\|_{L^2} \| v \|_{H^{\frac12+}} + \| u\|_{L^2} \|v\|_{H^{\frac12+}} \Bigr) \| n^{\pm j} \|_{L^2}
    \end{split}\\
    &\lesssim 
    \ni^{1-s} \nhl^{s-\frac12 -\beta  } \nm. 
\end{align*}

Thus the total growth of the Hamiltonian over $[0, \delta J]$ is bounded by 
\begin{align*}
\ni^{-1+} \nm^3 
+ (J-1) \Bigl( 
    \ni^{-1+} \nm^2  + 
    \ni^{-2+}\nm^4 + \ni^{-\frac32+}\nm^{3} + 
    \ni^{1-s} \nhl^{s-\frac12 -\beta   }\nm \Bigr)  \\
    = \nm^2 \Bigl[
    \ni^{-1+} \nm + 
    (J-1)\Bigl( 
    \ni^{-1+}   + 
    \ni^{-2+}\nm^2 + \ni^{-\frac32+}\nm  + \ni^{1-s} \nhl^{s-\frac12 -\beta   } \nm^{-1} \Bigr) \Bigl].  
\end{align*}
This is acceptable as long as $\ni^{-1+} \nm \lesssim 1$ and 
\[ J \lesssim \min\left\{ \ni^{1-}  ,\;\; \ni^{2-} \nm^{-2},\;\; \ni^{\frac32 -} \nm^{-1},\;\; \ni^{-(1-s)} \nhl^{-(s-\frac12 -\beta   )} \nm \right\} .\]
For such $J$, we can iterate to cover an interval of length 
\[ \delta J \approx  \min\left\{ \ni^{1-} \nm^{-2-}  ,\;\;  \ni^{2-} \nm^{-4-}, \;\; \ni^{\frac32 -} \nm^{-3-}, \;\; \ni^{-(1-s)} \nhl^{-(s-\frac12 -\beta   )} \nm^{-1-} \right\}. \] 
Recalling that 
\[ N \approx  \max\{ K\ni^{1-s} , C_1 \nhl^{\frac12 -\beta   }\}, \]
we obtain the following bounds. To simplify the expressions, here the implicit constants here depend on $K$, $C_1$, $s$, and $\beta$ (but not $T$):  
\begin{align*}
 \ni^{1-} \nm^{-2-} \gtrsim T \quad &\Leftrightarrow \quad \ni\gtrsim T^{\frac{1}{2s-1}+} \text{\quad\&\quad} \ni \gtrsim  T^{1+} \nhl^{1 - 2\beta    +}, 
\\  \ni^{2-} \nm^{-4-} \gtrsim T \quad &\Leftrightarrow \quad \ni \gtrsim  T^{\frac{1}{4s-2}+} \text{\quad\&\quad} \ni \gtrsim  T^{\frac12+}  \nhl^{1 - 2\beta    +}, \\
 \ni^{\frac32 -} \nm^{-3-} \gtrsim T \quad &\Leftrightarrow \quad \ni \gtrsim  T^{\frac{2}{6s-3}+} \text{\quad\&\quad} \ni \gtrsim  T^{\frac23+}  \nhl^{1 - 2\beta    +}.
 \end{align*}
 The final term in the minimum yields the most complicated constraint: 
 \begin{align*}
 \ni^{-(1-s)} \nhl^{-(s-\frac12 -\beta   )} \nm^{-1-} \gtrsim T \quad &\Leftrightarrow \quad \ni \lesssim  T^{\frac{-1+}{2(1-s)}}\nhl^{\frac{\beta + \frac12 - s}{2(1-s)}-} \text{\quad\&\quad} \ni \lesssim  T^{\frac{-1}{1-s}}  \nhl^{\frac{2\beta - s }{1-s}-}. 
 \end{align*} 

To satisfy all these constraints simultaneously, we require 
\[ s > \frac12 \]
and take
\[ \ni = \nhl^\alpha \text{ for some } \alpha \in \Biggl( 1 - 2\beta   , \min\Bigl\{\frac{\beta + \frac12 - s }{2(1-s)}, \; \frac{2\beta - s   }{1-s}  \Bigr\} \Biggr)\] 
and some $\nhl$ very large, dependent on $K$, $T$, $s$, and $r$.
This is possible as long as 
\[ 1- 2\beta    <  \min\Bigl\{\frac{\beta + \frac12 - s }{2(1-s)}, \; \frac{2\beta - s   }{1-s}  \Bigr\}. \]
Solving this with $s > \frac12$, we find the constraint
\[ \beta > \frac12 - \min\Bigl\{\frac{1-s}{4(1-s) +1}, \; \frac{1-s}{2(1-s) +2} \Bigr\}    = \frac12 -  \frac{1-s}{2(1-s) +2}     = \frac{1}{2(2-s)}   . \]

To obtain a polynomial bound, note that we may take 
\[\nhl = C(K, C_1, s, \beta) T^\gamma, \]
where 
\[ \gamma = \max\Bigr\{ \frac{1}{(2s-1)\alpha}, \frac{1}{\beta + \frac12 - s - 2\alpha(1-s)}, \frac1{\alpha + 2\beta -1   },  \frac{1}{2\beta - s    - \alpha(1-s)} \Bigl\}+.  \]
Therefore we can take
\[ N \approx \biggl( T^{\alpha(1-s) \gamma} + T^{\bigl(\frac12-\beta   \bigr) \gamma} \biggr). \] 
This allows us to conclude that the nonlinear part of the Zakharov flow grows at most polynomially in the $H^s \times L^2$ norm. 

\end{proof}

\section{Local Theory Result}

The section contains the proof of the required local theory for the difference equation. The statement is as follows.

\begin{prop} \label{localtheory}
Fix $\delta \ll 1$, $\epsilon_0 \ll 1$, $C_0 > 0$, and input functions $(u,n)$ such that 
\begin{equation*}
\begin{cases}
 \delta^{\frac12-\epsilon_0}\|u\|_{X^{s,\frac12,1}_{\delta}} &\leq C_0 \\
\delta^{\frac12-\epsilon_0}\| n\|_{Y^{0,\frac12,1}_{\delta}} &\leq C_0. 
\end{cases}
\end{equation*}
Consider the difference equation on $[0,\delta]$ given by
\begin{equation}\label{eq:diff_eq}
\begin{cases}
iv_t + \Delta v = \Re(n + m)(u + v) - \Re(n)u \\
im_t - Dm = D[ |u + v|^2 - |u|^2 ] 
\end{cases}
\end{equation}
with initial data
\begin{align*}
v(x,0)  &= 0 \\
m(x,0) &= m_0 := \W(t_0) \Biggl( \sum_{|k| \geq \nhl} \frac{h_k}{\lb k \rb^\beta} e^{ikx}\Biggr).
\end{align*}
We assume that the coefficients $h_k$ satisfy
\[ \sup_k |h_k| \leq C_1 \quad \text{for some} \quad C_1 >0. \] 

We further assume that $s - \frac12- \beta     < 0$. Then for $\nhl \gg 1$ sufficiently large and $\delta$ sufficiently small, the difference equation \eqref{eq:diff_eq}
has a solution in $H^s \times H^{\beta - \frac12  -}$.  Furthermore, if we write 
\begin{align*}
m(x,t) &= \W(t) m_0(x) + \t{m}(x,t),
\end{align*}
then we have for $t \in [0,\delta]$
\begin{equation*}
\| v \|_{H_x^s} + \| \t{m}\|_{ L^2_x} \lesssim \nhl^{s-\beta - \frac12   } \ll 1. 
\end{equation*}
\end{prop}

The proof of this amounts to proving a contraction for the difference equation
\begin{equation}
\begin{cases}
i{v} + \Delta v = \Re(n + \t{m} + \W(t)m_0)(u + v ) - \Re(n) u, \\
i\t{m}_t - D\t{m} = D[|u + v |^2 - |u|^2] 
\end{cases}
\end{equation}
with zero initial data on the interval $[0,\delta]$ in a ball of radius $\approx \nhl^{s -\frac12 -\beta   }$ in the space $X^{s,\frac12, 1}_{\delta} \times Y^{0,\frac12, 1}_{\delta}$. 

The only problematic term to estimate is 
\[ \Re(\W(t) m_0)(u + v). \]
All others are covered by existing bilinear estimates due to Kishimoto \cite{Kish1} as long as we take $\delta$ sufficiently small. Furthermore, since $u$ typically has much larger norm that $v$, it suffices to consider 
\[ \Re(\W(t) m_0)u. \]
To close the contraction, it suffices to obtain a bound of the form 
\[\left\|  \eta_\delta(t) \int_0^t e^{i(t-t')\Delta}\Bigl(\Re(\W(t) m_0)u \Bigr) \d t' \right\|_{X^{s,\frac12,1}_\delta} \lesssim C_1 \delta^{\frac12-}  \nhl^{s-\frac12-\beta  } \| u\|_{X^{s,\frac12,1}_\delta} \lesssim C_0 C_1 \delta^{\epsilon_0 -} \nhl^{s- \frac12 -\beta   }.  \]  
The remainder of the paper is devoted to obtaining this bound, relying heavily on estimates previously established by Kishimoto. 

The requirement that $\delta$ be small exists because we will require
\begin{align*}
 C_0 \delta^{\epsilon_0 -} \lesssim 1 \\
 C_0 C_1 \delta^{\epsilon_0-} \lesssim 1
\end{align*}
to close the contraction.

In the following, we employ dyadic decompositions. We always use $P_{N_1,L_1}  \Re(\W^\pm(t) m_0^\pm)u$, $P_{N_2,L_2}u$, and $Q^\pm_{N_0,L_0} \W^\pm m_0^\pm$; i.e. the product $\Re(\W^\pm(t) m_0^\pm)u$ is associated with the dyadic variables $N_1$ and $L_1$, the linear wave flow is associated the dyadic variables $N_0$ and $L_0$, etc.

We also use bars to denote maxima and minima of the dyadic variables, i.e.
\[ \overline{L}_{jk} := \max\{ L_j, L_k\}, \qquad \overline{L} := \overline{L}_{012} = \max\{ L_0, L_1, L_2\},\]
\[ \underline{L}_{jk} = \min\{ L_j, L_k\}, \qquad \underline{L} = \underline{L}_{012} = \min\{ L_0, L_1, L_2\}.\]
We also define
\[ L_m := \operatorname{median}\{ L_0, L_1, L_2\}.\] 

\subsection{Resonant case}

Here we confine our attention to the resonant frequencies. These occur when 
\[ 0= |L_1-L_2-L_0| \lesssim \overline{L}.\] 
We have 
\[ |L_1-L_2-L_0| =|\tau_0 + \tau_2 - (k_0 + k_2)^2 - \tau_2 + k_2^2 - \tau_0 \mp |k_0||  = |k_0| | k_0 + 2k_2 \pm \lambda(k_0)|,   \]
where $\lambda(k)$ is the sign of $k$. 
Since $k_0 \neq 0$, solving this yields
\[ k_0 = 2k_1 \pm \lambda(k_1) \qquad \qquad k_2 = \mp\lambda(k_1) - k_1, \]
for $k_1 \neq 0.$ 
 
On these frequencies, $\W^\pm(t) m_0^\pm u$ can be estimated in $X^{s,-\frac12,1}_\delta$ as follows. We use the fact that $|h_k^\pm| \leq C_1 $, and compute 
\begin{align*}
&\| \W^\pm(t) m_0^\pm u \|_{X^{s,-\frac12, 1}_{\delta}} \\ 
\lesssim
&\left\| N_1^s L_1^{-\frac12} \left\| \int  \frac{h^\pm_{2k_1 - \lambda(k_1)}}{\lb k_1 \rb^\beta } \hat{u}(\mp\lambda(k_1) - k_1, \tau_2) \delta \hat{\eta}(\delta(\tau_1 - \tau_2 \pm | 2k_1 - \lambda(k_1)|)) \d \tau_2  \right\|_{\ell^2_{k_1}(|k_1| \approx N_1)L^2_{\tau_1}(|\tau_1 - k_1^2| \approx L_1) }\right\|_{\ell^2_{N_1}\ell^1_{L_1}} \\
\lesssim 
& C_1\delta^{\frac12- } \Bigg\| N_1^{-\beta   } L_1^{0-} \Bigg\| \int  \lb \tau_1 - k_1^2\rb^{-\frac12+} \frac{\lb \tau_2 - (\mp \lambda(k_1) -k_1)^2 \rb^{-\frac12}}{\lb \tau_1 -\tau_2 \pm | 2k_1 \pm \lambda(k_1)| \rb^{\frac12+}} \times \\
 & \hspace{2.5in} \left( \sum_{L_2}N_1^s L_2^\frac12 \hat{P_{N_1,L_2}u}(\mp \lambda(k_1) - k_1, \tau_2)  \right) \d \tau_2 \Bigg\|_{\ell^2_{k_1}(|k_1| \approx N_1)L^2_{\tau_1}(|\tau_1 - k_1^2| \approx L_1) } \Bigg\|_{ \ell^2_{N_1} \ell^1_{L_1}} \\
\lesssim 
& C_1\delta^{\frac12- } \left\| N_1^{-\beta   } L_1^{0-} \left\|  \lb \tau_1 - k_1^2\rb^{-1+} \left\| \sum_{L_2}N_1^s L_2^\frac12 \hat{P_{N_1,L_2}u}(\mp\lambda(k_1) - k_1, \tau_2)  \right\|_{L^2_{\tau_2}} \right\|_{\ell^2_{k_1}(|k_1| \approx N_1)L^2_{\tau_1}(|\tau_1 - k_1^2| \approx L_1) } \right\|_{ \ell^2_{N_1} \ell^1_{L_1}} \\
\lesssim 
& C_1\delta^{\frac12- } \left\| N_1^{-\beta   } L_1^{0-} \left\|  \left\| \sum_{L_2}N_1^s L_2^\frac12 \hat{P_{N_1,L_2}u}(\mp\lambda(k_1) - k_1, \tau_2)  \right\|_{L^2_{\tau_2}} \right\|_{\ell^2_{k_1}(|k_1| \approx N_1 } \right\|_{ \ell^2_{N_1} \ell^1_{L_1}} \\
\lesssim 
& C_1\delta^{\frac12- } \left\| N_1^{-\beta   } L_1^{0-}  \sum_{L_2}N_1^s L_2^\frac12 \left\| P_{N_1,L_2}u \right\|_{L^2L^2}   \right\|_{ \ell^2_{N_1} \ell^1_{L_1}} \\
\lesssim 
& C_1\delta^{\frac12- } \left\| N_1^{-\beta   }  \sum_{L_2}N_1^s L_2^\frac12 \left\| P_{N_1,L_2}u \right\|_{L^2L^2}   \right\|_{ \ell^2_{N_1}} \\
\lesssim 
& C_1\delta^{\frac12- } \nhl^{-\beta   } \| u \|_{X^{s,\frac12,1}_I} \lesssim C_0 C_1 \delta^{\epsilon_0-} \nhl^{-\beta   }.
\end{align*}

\subsection{High Schr\"odinger frequencies}
We decompose dyadically in frequency space as follows. For general $m$,
\begin{align*}
&\| \eta_\delta \int_0^t S(t - t')(mu) d t' \|_{X^{s,\frac12, 1}_I} \\
&\approx
\left[ \sum_{N_1} \| P_{N_1} \eta_\delta \int_0^t S(t - t')(mu) d t' \|^2_{X^{s,\frac12,1}_I} \right]^\frac12 \\
&\lesssim 
\left[ \sum_{N_1} \left( \sum_{N_0, N_2} \sum_{L_0, L_1, L_2} \left\| \eta_\delta \int_0^t S(t-t') P_{N_1, L_1}[P_{N_2,L_2}(\eta_\delta u) Q_{N_0,L_0}(\eta_\delta m) ] d t' \right\|_{X^{s,\frac12,1}_I} \right)^2 \right]^\frac12.
\end{align*}

Furthermore, we calculate that 
\begin{align*}
\| Q_{N_0,L_0}( \eta_\delta \W(t) m_0) \|_{L^2L^2}  &= 
\left\| \frac{h_k}{\lb k \rb^\beta}  \delta \hat{\eta}(\delta(\tau \pm |k|)) \right\| _{L^2_\tau \ell^2_k(|k| \approx N_0, |\tau \pm |k| \approx L_0)}\\
&= 
\delta \left\| \frac{h_k}{\lb k \rb^\beta}  \right\|_{\ell^2(|k| \approx N_0)} \left\| \hat{\eta}(\delta \tau) \right\|_{L^2(|\tau| \approx L_0)} \\
&\lesssim
C_1 \delta \left( N_0^{\frac12 - \beta   } \right) L_0^\frac12 \lb \delta L_0 \rb^{-a}  = C_1 \delta N_0^{\frac12 - \beta   }  L_0^\frac12 \lb \delta L_0 \rb^{-a}.
\end{align*}

\subsubsection{$N_0 \approx N_2 \gg N_1$ and $\overline{L_{02}} \gtrsim N_0^2$}
 Using Kishimoto's Lemma 4.1 and Corollary 3.3, 
\begin{align*}
 &\left[ \sum_{N_1} \left( \sum_{N_0, N_2} \sum_{L_0, L_1, L_2} \left\| \eta_\delta \int_0^t S(t-t') P_{N_1, L_1}[P_{N_2,L_2}(\eta_\delta u) Q_{N_0,L_0}(\eta_\delta w) ] d t' \right\|_{X^{s,\frac12,1}_I} \right)^2 \right]^\frac12 \\
 &\lesssim C_1
 \left[ \sum_{N_1} \left( \sum_{N_0, N_2} \sum_{L_0, L_1, L_2} \delta^{\frac 32 -b } N_1^s L_1^{-b} \overline{L}^\frac12 L_m^{\frac14+} \underline{L}^{\frac14 +} N_1^{\frac12-} N_2^{-1} N_0^{\frac12 - \beta   }  L_0^\frac12 \lb \delta L_0 \rb^{-a} \|P_{N_2,L_2}(\eta_\delta u)\|_{L^2L^2}  \right)^2 \right]^\frac12 \\
 &\lesssim C_1
 \left[ \sum_{N_1} \left( \sum_{N_2} \sum_{L_0, L_1, L_2} \delta^{\frac 32 -a -b } N_1^{s+ \frac12 -} L_1^{-b} \overline{L}^\frac12 L_m^{\frac14+} \underline{L}^{\frac14 +} N_2^{-\frac12 - \beta   }  L_0^{\frac12 -a}  \|P_{N_2,L_2}(\eta_\delta u)\|_{L^2L^2}  \right)^2 \right]^\frac12.
\end{align*}
If $L_2 \geq L_0$, then $\overline{L} \lesssim L_2$, so we have the bound
\[ C_1
\left[ \sum_{N_1} \left( \sum_{N_2} \sum_{L_0, L_1, L_2} \delta^{\frac 32 -a -b }  L_0^{\frac34 -a +} L_1^{\frac14 -b+} N_1^{s + \frac12 -}N_2^{-s - \frac12 - \beta   }   \left( N_2^s L_2^\frac12 \|P_{N_2,L_2}(\eta_\delta u)\|_{L^2L^2} \right) \right)^2 \right]^\frac12.
\]
Take $a = \frac34 +$ and $b = \frac14 +$ to obtain
\begin{align*} 
C_1&\left[ \sum_{N_1} \left( \sum_{N_2} \sum_{ L_2} \delta^{\frac12 -} N_1^{s +\frac12 -}N_2^{-s- \frac12 - \beta   }   \left( N_2^s L_2^\frac12 \|P_{N_2,L_2}(\eta_\delta u)\|_{L^2L^2} \right) \right)^2 \right]^\frac12 \\
\lesssim
C_1 &\delta^{\frac12- } \| u\|_{X^{s,\frac12,1}} \left[\sum_{N_2 \gg N_1} N_1^{2s + 1 -}N_2^{-2s - 1- 2\beta } \right]^\frac12 \\
\lesssim 
C_1 &\delta^{\frac12- } \nhl^{-\beta   } \| u\|_{X^{s,\frac12,1}} \lesssim C_0C_1 \delta^{\epsilon_0-} \nhl^{-\beta   }.
\end{align*}

If $L_2 \lesssim L_0$, then $\overline{L} \lesssim L_0$ and we have the bound
\[  C_1 \left[ \sum_{N_1} \left( \sum_{N_2} \sum_{L_0, L_1, L_2} \delta^{\frac 32 -a -b } L_0^{1-a} L_1^{\frac14 - b +}  N_1^{s +\frac12 -}  N_2^{-s -\frac12 - \beta   }  \left( N_2^s L_2^{\frac14+} \|P_{N_2,L_2}(\eta_\delta u)\|_{L^2L^2} \right) \right)^2 \right]^\frac12.\] 
Take $a = 1+$ and $b = \frac14+$ to obtain
\begin{align*}
C_1 \delta^{\frac14- } \nhl^{-\beta  } \| u\|_{X^{s,\frac14+,1}} \lesssim \delta^{\frac12- } \nhl^{-\beta    } \| u\|_{X^{s,\frac12,1}} \lesssim C_0C_1 \delta^{\epsilon_0-} \nhl^{-\beta   }.
\end{align*}

\subsubsection{$N_0 \approx N_2 \gg N_1$ and $\overline{L_{02}} \ll N_0^2$}

Using the same results as in the previous case, and noting that $L_1 = \overline{L} \approx N_0^2$, we have the bound

\begin{align*}
&C_1\left[ \sum_{N_1} \left( \sum_{N_2} \sum_{L_0, L_1, L_2} \delta^{\frac 32- a-b } L_0^{\frac34 - a +} L_1^{\frac12 -b}   N_1^{s +\frac12 - }N_2^{-s-\frac12 - \beta   }   \left( N_2^s L_2^{\frac14+} \|P_{N_2,L_2}(\eta_\delta u)\|_{L^2L^2} \right) \right)^2 \right]^\frac12.
\end{align*}
Take $b = \frac12$ and $a = \frac34 +$ and proceed as above to obtain the bound
\begin{align*}
C_1 \delta^{\frac14- } \nhl^{-\beta  } \| u\|_{X^{s,\frac14+,1}} \lesssim C_1 \delta^{\frac12- } \nhl^{-\beta   } \| u\|_{X^{s,\frac12,1}} \lesssim C_0C_1 \delta^{\epsilon_0 -} \nhl^{-\beta   } .
\end{align*}

\subsubsection{$N_1 \approx N_2$, $\overline{L_{02}} \ll N_1^2$, and $L_1 \gtrsim N_1^2$}

In this case, we in fact have $\overline{L} = L_1 \approx N_1^2$. Also note that $N_0 \lesssim N_0, N_1$. We use Kishimoto's Proposition 3.1. This gives the bound
\begin{align*}
&  C_1 \sum_{N_0, N_2} \sum_{L_0, L_1, L_2} \delta^{\frac 32- a-b } L_0^{\frac34 - a } L_1^{\frac12 -b}   N_0^{1 - \beta    } N_2^{-1 }   \left( N_2^s L_2^{\frac14} \|P_{N_2,L_2}(\eta_\delta u)\|_{L^2L^2} \right).
\end{align*}
Here we take $a = \frac34 +$ and $b = \frac12$ and again obtain
\[C_1 \delta^{\frac12- } \nhl^{-\beta   } \| u\|_{X^{s,\frac12,1}} \lesssim 
C_0C_1 \delta^{\epsilon_0 -} \nhl^{-\beta   }.\]

\subsubsection{$N_1 \approx N_2$ and $\overline{L_{02}} \gtrsim N_1^2$}

Here $N_0 \lesssim N_1$ and $L_1 \lesssim L_{m}$. We use Kishimoto's Prop. 3.1 again.

If $L_2 \geq L_0$, then we arrive at
\begin{align*}
& C_1 \sum_{N_0, N_2} \sum_{L_0, L_1, L_2} \delta^{\frac 32- a-b } L_0^{\frac34 - a } L_1^{\frac14 -b}   N_0^{1 - \beta   } N_2^{-1 }   \left( N_2^s L_2^{\frac12} \|P_{N_2,L_2}(\eta_\delta u)\|_{L^2L^2} \right).
\end{align*}
This closes by taking $a = \frac34 +$ and $b = \frac14 +$.
Otherwise $L_0 \geq L_2$ and we arrive at 
\begin{align*}
& C_1 \sum_{N_0,N_2} \sum_{L_0, L_1, L_2} \delta^{\frac 32- a-b } L_0^{1 - a } L_1^{\frac14 -b}   N_0^{1 - \beta    } N_2^{-1 }   \left( N_2^s L_2^{\frac14} \|P_{N_2,L_2}(\eta_\delta u)\|_{L^2L^2} \right) .
\end{align*}
This closes by taking $a = 1 +$ and $b= \frac14 +$. 

\subsubsection{$N_1 \approx N_2$ and $\overline{L} \ll N_1^2$.}

Here, since the resonant case has already been addressed, we also have $\overline{L} \gtrsim N_1$, and hence we can apply the proof of Kishimoto's Prop. 3.5. 
If $L_0 = \overline{L}$ with $\overline{L_{12}} \gtrsim N_0$, or if $\overline{L} = L_1$ or $L_2$, with $L_{m} \gtrsim N_1$ or $L_{m} \ll N_1  \lesssim \overline{L}/N_0$, Kishimoto's proof translates directly to the one-dimensional case, and we get
\begin{align*}
&C_1 \sum_{N_0, N_2} \sum_{L_0, L_1, L_2} \delta^{\frac 32- a-b } L_0^{\frac12 - a } L_1^{ -b} L_2^{-c} \underline{L}^\frac14 L_m^{\frac38} \overline{L}^{\frac38}    N_0^{ - \beta   }   \left( N_2^s L_2^c \|P_{N_2,L_2}(\eta_\delta u)\|_{L^2L^2} \right).
\end{align*}
To close this, we need $a + b + c = \frac32+$, and each dyadic sum in $L_i$ to converge. Note that $N_2 \lesssim \overline{L}$, so the sum in $N_2$ contributes $\overline{L}^{0}+$. 

If $\underline{L} = L_0$, then take $a = \frac34 +$, and $b=c=\frac38+$. 

If $\underline{L} = L_1$, we can take $a = \frac78+$, $b = \frac14+$, and $c = \frac38+$. 

If $\underline{L} = L_2$, then take $a = \frac78+$, $b = \frac38+$, and $c = \frac14+$. 

These lead to the bound
\[ C_0C_1 \delta^{\epsilon_0 -} \nhl^{-\beta    } \] 
as desired. 

Otherwise, if $\overline{L} = L_0$ and $L_m \lesssim N_0$, we note from Kishimoto's proof that for fixed $k_2$, the frquency $k_1$ is confined to an interval of length $\lesssim \overline{L}/N_1$. This leads to the bound
\begin{align*}
&\sum_{N_0, N_2} \sum_{L_0, L_1, L_2} \delta^{\frac 32- a-b } L_0^{\frac12 - a } L_1^{ -b} L_2^{-c} \underline{L}^\frac12   \overline{L}^\frac12    N_0^{\frac12  - \beta   }  N_2^{-\frac12} \left( N_2^s L_2^c \|P_{N_2,L_2}(\eta_\delta u)\|_{L^2L^2} \right).
\end{align*}
This can be bounded by $\delta^{\frac12-} \nhl^{-\beta   } \| u\|_{X^{s,\frac12,1}_\delta}$ by taking $a = 1+$, $b = \frac12+$, and $c= 0+$.  

The final possibility is $\overline{L} = L_1$ or $L_2$ and $L_m \ll N_1$ with $\overline{L} \ll N_0 N_1$. Exactly the same argument can be used to treat this case.

\subsection{Low Schr\"odinger Frequencies}

Here we assume that $N_2 \ll N_0 \approx N_1$. Noting that $\overline{L} \gtrsim N_0^2$, we can use Kishimoto's Prop. 3.1, which carries through for dimension one, to obtain
\begin{align*}
&\sum_{N_0, N_2} \sum_{L_0, L_1, L_2} \delta^{\frac32- a-b  } L_0^{\frac12 - a } L_1^{ -b} L_2^{-c} \underline{L}^\frac14 L_m^\frac14 \overline{L}^{\frac12}  N_0^{s - \frac12 - \beta   } N_2^{\frac12-s}   \left( N_2^s L_2^c \|P_{N_2,L_2}(\eta_\delta u)\|_{L^2L^2} \right).
\end{align*}
 
This closes by taking $a + b + c = \frac32 +$ as in the previous case and yields the desired bound 
\[ \delta^{\frac12-} \nhl^{s- \frac12 - \beta   } \| u\|_{X^{s,\frac12,1}_\delta} \lesssim C_0C_1 \delta^{\epsilon_0 -} \nhl^{s- \frac12 - \beta   }. \]

\section{Proof of Proposition \ref{growth_bound}} \label{prf_of_growth_bound}

The proof of this proposition is almost identical to that of \cite[Prop. 4.1]{Kish2}. Note that the only difference between our result and that of Kishimoto is that we replace an factor of $\ni^{-\frac54 +} \delta^{\frac14 -}$ in the estimate by a factor of $\ni^{-\frac32 +}\delta^{\frac12-}$.

We can thus repeat the proof of \cite[Prop. 4.1]{Kish2} verbatim, except for the estimate of term (4.2), Cases 1(ii) and 2(ii). These are the cases which lead to the factor $\ni^{-\frac54 +} \delta^{\frac14 -}$. In those cases, an $L^4_{x,t}$ Strichartz estimate \cite[Lemma 2.11]{Kish2} was used to obtain the bounds. To prove our result, it suffices to obtain a one-dimensional analogue of this Strichartz estimate. Specifically, we wish to show the following.
\begin{prop}
 For $u$, $v \in L^2(\T \times \R)$, we have \[ \| nm \|_{L^2_{x,t}} \lesssim L^\frac12 N^\frac12 \| n \|_{L^2_{x,t} }\|m\|_{L^2_{x,t}},\]
 where for some fixed $N$ and $L$ dyadic,
 \[ \operatorname{supp}\t{n}(k ,\tau), \; \operatorname{supp}\t{m}(k ,\tau) \subseteq \{ |k| \approx N\} \cap \{ |\tau \pm |k|| \approx L\}. \] 
\end{prop}

\begin{proof}
	We have, using Cauchy-Schwartz and Young's inequalities,
	\begin{align*}
	\| nm \|_{L^2_{x,t}} 
	&= \| \hat{nm}\|_{L^2_{\tau} \ell^2_k} = \| \hat{n} \star \hat{m}\|_{L^2_{\tau} \ell^2_k}\\
	&\leq \sup_{|k| \approx N, \tau} | B(k, \tau)|^{\frac12} \| \hat{n}\| _{L^2_\tau\ell^2_k}\| \hat{m}\|_{L^2_\tau\ell^2_k},
	\end{align*}
where 
\[ B(k,\tau) = \Bigl\{ (k_1,\tau_1) \; \Big| \; | k_1|,\, |k-k_1| \approx N, \; |\tau_1 \pm |k_1||,\, | \tau - \tau_1 \pm |k-k_1|| \approx L \Bigr\}. \] 
We can bound the size of the set by noting that $\tau_1$ can range in an interval of length at most $L$ and $k_1$ can range over an interval of size $N$. This gives 
\begin{align*}
|B(k,\tau) | \leq LN.
\end{align*}
Inserting this bound on $|B|$ into the estimate above gives the desired result. 
\end{proof}
Applying this Strichartz estimate to a frequency-constrained function $n$, we obtain
\[ \| n \|_{L^4_{x,t}} \lesssim N^\frac14 L^\frac14 \| n\|_{L^2_{x,t}} \approx N^\frac14 \| n\|_{Y^{0,\frac14, 1}_{\pm}}. \]
Note that in the $\T^2$ case, Kishimoto instead obtained 
\[ \| n \|_{L^4_{x,t}} \lesssim N^\frac38 \| n\|_{Y^{0,\frac38, 1}_{\pm}}. \]
Lowering the power $\frac38$ to $\frac14$ saves us a factor of $N^\frac18$ and gives us an additional power of $\delta^\frac18$. Since the $L^4_{x,t}$ estimate is used on two terms in a product, the net result is gain of $N^{-\frac14}\delta^{\frac14} $ in the estimate, as desired.


\end{document}